\documentclass[12pt,a4paper]{article}

\usepackage[dvips]{graphics}
\usepackage[dvips]{epsfig}

\newtheorem{theorem}{Theorem}

\newtheorem{lemma}[theorem]{Lemma}
\newtheorem{proposition}[theorem]{Proposition}
\newtheorem{remark}{Remark}
\newenvironment{proof}[1][Proof]{\textbf{#1.} }{\ \rule{0.5em}{0.5em}}

\usepackage{latexsym}
\usepackage{amsfonts}
\usepackage{amssymb}
\usepackage{amsmath}
\usepackage[english]{babel}

\def\G{\Gamma}
\def\g{\gamma}

\begin{document}

\title{A census of genus two 3-manifolds up to 42 coloured tetrahedra
\thanks{Work performed under the auspicies of the G.N.S.A.G.A.
of I.N.D.A.M. (Istituto Nazionale di Alta Matematica) and
financially supported by MiUR of Italy (project ``Propriet\`a
geometriche delle variet\`a reali e complesse'') and by the
University of Modena and Reggio Emilia (project ``Modelli discreti
di strutture algebriche e geometriche'').}}

\author{Paola Bandieri,\ Paola Cristofori \ and Carlo Gagliardi}

\maketitle

\begin{abstract}
We improve and extend to the non-orientable case a recent result
of Karabas, Ma\-licki and Nedela concerning the classification of
all orientable prime $3$-manifolds of Heegaard genus two,
triangulated with at most $42$ coloured tetrahedra.
\\
\\
\noindent\textit{2000 Mathematics Subject Classification:} 57N10,
57Q15, 57M15.\\\textit{ Keywords:} genus two 3-manifold;
crystallization; edge-coloured graph.
\end{abstract}

\section{Introduction}

In \cite{KMN}, Karabas, Malicki and Nedela show that there exist
exactly $78$ non-homeomorphic, closed, orientable, prime
$3$-manifolds with Heegaard genus two, admitting a coloured
triangulation with at most $42$ tetrahedra.

Each manifold $M$ is identified by a suitable $6$-tuple of
non-negative integers, representing a minimal crystallization --
hence a minimal coloured triangulation -- of $M$. From such a
$6$-tuple, a presentation of the fundamental group and of the
first homology group of $M$ are easily obtained (see also
\cite{KMN1}).

The result is performed first by generating all ``admissible"
$6$-tuples, en\-co\-ding genus two crystallizations up to order
$42$ (\cite{KN}) and then by using combinatorics, topology and
group
theory to subdivide them into $78$ equivalence classes (after 
excluding $\mathbb S^3$, $\mathbb S^1 \times \mathbb S^2$, lens
spaces and connected sums), which are proved to be in one-to-one
correspondence with the homeomorphism classes of the represented
$3$-manifolds.

In the present paper, we improve the previous result and extend it
to the non-orientable case, by using a computer program which
generates directly all (bipartite and non-bipartite) $3$-manifold
crystallizations of a given order.

The procedure, restricted to graphs of regular genus two and order
at most $42$, produces as output $703$ bipartite crystallizations
(thus representing orientable $3$-manifolds) and $82$
non-bipartite crystallizations (thus representing non-orientable
$3$-manifolds).

A classification algorithm based on the concept of ``dipole
moves", implemented in a C++ program, enables us to partition the
graphs previously generated into $175$ classes in the bipartite
case and into $9$ classes in the non-bipartite case, which are
proved to represent non-homeomorphic (orientable and
non-orientable) $3$-manifolds, of Heegaard genus $\leq 2$ (with
the given bound for the number of vertices of the
crystallizations).

In the orientable case, $97$ classes represent genus one
$3$-manifolds or connected sums. The remaining $78$ classes,
representing prime, genus two $3$-manifolds, are listed in Table
$2$, by increasing number of vertices of the crystallizations,
 where for each class a geometric description, the
representative $6$-tuple and the position in
Karabas-Malicki-Nedela's list are presented. This completes the
identification of all still unknown manifolds of \cite{KMN}.

In the non-orientable case, two classes represent $\mathbb
S^1\tilde\times\mathbb S^2$ (the twisted $2$-sphere bundle over
$\mathbb S^1$, of Heegaard genus one) and a connected sum
respectively. Hence, there exist exactly $7$ prime, non-orientable
$3$-manifolds with genus two, all listed and identified in Table
$3$, again by increasing number of vertices of the
crystallizations.

\section{Preliminaries}

Throughout this paper, spaces and maps will be in 
PL-category, for which we refer to \cite{RS}. Manifolds will  be
closed and connected, when not otherwise specified. The symbol
$\cong$ will mean \textit{PL-homeomorphism}.

Crystallization theory provides an useful tool for representing
manifolds by means of edge-coloured graphs (\cite{P}). In this
section, we limit ourselves to give definitions and results which
are necessary to understand our work. For an exhaustive look on
the theory, we refer to \cite{BCG} and \cite{FGG}. For the basic
facts about graph theory see \cite{GY}.

An \textit{$(n+1)$-coloured graph} is a pair $(\G,\gamma)$, where
$\G$ is a graph, regular of degree $n+1$, and $\gamma :
E(\G)\to\Delta_n=\{0,\ldots,n\}$ a map which is injective on each
pair of adjacent edges of $\G$. In the following, we will often
write $\G$ instead of $(\G,\gamma)$.

For each $B\subseteq\Delta_n$, we call $B$\textit{-residues} of
$(\G,\gamma)$ the connected components of the coloured graph
$\G_B=(V(\G),\gamma^{-1}(B))$; given an integer
$m\in\{1,\ldots,n\}$, we call $m$\textit{-residue} of $\G$ each
$B$-residue of $\G$ with $\# B=m$. Moreover, for each
$i\in\Delta_n$, we set $\hat{\imath}=\Delta_n\setminus\{i\}$.

An isomorphism $\phi : \G\to\G'$ is called a \textit{coloured
isomorphism} between the $(n+1)$-coloured graphs $(\G,\gamma)$ and
$(\G',\gamma')$ if there exists a permutation $\varphi$ of
$\Delta_n$ such that $\varphi\circ\gamma=\gamma'\circ\phi$.

A coloured $n$-complex is a pseudocomplex $K$ of dimension $n$
with a labelling of its vertices by $\Delta_n=\{0,\ldots,n\}$,
which is injective on the vertex-set of each simplex of $K$.

For each $(n+1)$-coloured graph $\G$, a coloured $n$-complex
$K(\G)$ can be obtained by the following rules:
\begin{itemize}
\item [-] for each vertex $v$ of $\G$, take an $n$-simplex $\sigma (v)$ and label
its vertices by $\Delta_n$;
\item [-] if $v$ and $w$ are vertices of $\G$ joined by an $c$-coloured
edge ($c\in\Delta_n$), then identify the $(n-1)$-faces of $\sigma
(v)$ and $\sigma (w)$ opposite to the vertices labelled $c$.
\end{itemize}

If $M$ is a manifold of dimension $n$ and $\G$ an $(n+1)$-coloured
graph such that $|K(\G)|\cong M$ (here $|K(\G)|$ denotes the space
of the complex $K(\G)$), then, following Lins (\cite{L2}), we say
that $\Gamma$ is a \textit{gem} (graph-encoded-manifold)
representing $M.$

If, for each $i\in\Delta_n$, $\Gamma_{\hat{\imath}}$ is connected
(equivalently if the corresponding coloured triangulation
$K(\Gamma)$ has exactly one vertex labelled $i$, for each
$i\in\Delta_n$), then  $\Gamma$ and $K(\Gamma)$ are called
\textit{contracted}; furthermore, a contracted gem representing an
$n$-manifold $M$ is called a \textit{crystallization} of $M$. Note
that $M$ is orientable iff $\G$ is  bipartite.

Given two $(n+1)$-coloured graphs $\G^\prime$ and
$\G^{\prime\prime}$ representing the manifolds $M^\prime$ and
$M^{\prime\prime}$ respectively, we can easily construct an
$(n+1)$-coloured graph $\G=\Gamma^\prime\#\Gamma^{\prime\prime}$
representing $M^\prime\# M^{\prime\prime}$. Let $x$ be a vertex of
$\G^\prime$ and $y$ a vertex of $\G^{\prime\prime}$; then we
obtain $\G$ by removing $x$ from $\G^\prime$ , $y$ from
$\G^{\prime\prime}$ and by gluing the ``hanging" edges according
to their colours (see \cite{FGG}).

It is well-known that, if both manifolds are orientable (i.e.
$\G^\prime$ and $\G^{\prime\prime}$ are both bipartite) and do not
admit orientation-reversing automorphisms, there exist two
non-homeomorphic connected sums. In this case, by the above
construction, we can obtain two $(n+1)$-coloured graphs, each
corresponding to fix $x$ in $V(\G^\prime)$ and choose $y$ in one
of the two different bipartition classes of
$V(\G^{\prime\prime})$.

Let $\G$ be an $(n+1)$-coloured graph representing an $n$-manifold
$M$ and suppose that $\G$ satisfies the following condition (which
in the following will be referred to as ``condition (\#)''):

\begin{itemize}
\item [(\#)] $\G$ has $n+1$ edges
$\{e_0,\ldots,e_n\}$, one for each colour $i \in\Delta_n$, such
that $\G - \{e_0,\ldots,e_n\}$ splits into two connected
components.
\end{itemize}

\noindent Then it is easy to reverse the connected sum
construction and, starting from $\G$, obtain two $(n+1)$-coloured
graphs $\G^\prime$ and $\G^{\prime\prime}$, representing two
$n$-manifolds $M^\prime$ and $M^{\prime\prime}$ respectively, such
that $\G = \G^\prime \# \G^{\prime\prime}$. Hence $\G$ represents
$M^\prime \# M^{\prime\prime}$ (more precisely, $\G$ represents
one of the two possibly non-homeomorphic connected sums).

Coloured graphs appearing in our catalogues are always represented
by a numerical ``string", which is called the \textit{code}; it
describes completely the combinatorial structure of the coloured
graph (see \cite{CG} for definition and description of the related
\textit{rooted numbering algorithm}) and, since two
$(n+1)$-coloured graphs are colour-isomorphic iff they have the
same code (\cite{CG}), by re\-pre\-sen\-ting each coloured graph
by its code, we can easily reduce any catalogue of
crystallizations to one containing only non-colour-isomorphic
graphs. Moreover the code is easy to be handled by computer and
starting from a code $c$ there is a standard way to contruct a
coloured graph having $c$ as its code.

Main tools of our work are combinatorial moves (\textit{dipole
moves}) which tran\-sform a gem representing an $n$-manifold into
another (usually non-colour-isomorphic) gem, representing the same
manifold.

If $x,y$ are two  vertices of a $(n+1)$-coloured graph $(\G,\g )$
joined by $k$ edges $\{e_1,\ldots, e_k\}$ with $\g(e_h)=i_h$, for
$h = 1, \ldots, k$, then we call $\theta=\{x,y\}$ a
\textit{k-dipole} or a \textit{dipole of type k} in $\G$,
\textit{involving colours} $i_1,\ldots, i_k$, iff $x$ and $y$
belong to different $(\Delta_n - \{ i_1,\ldots, i_k\})$-residues
of $\G$.

In this case a new $(n+1)$-coloured graph $(\G^\prime,\g^\prime )$
can be obtained from $\G$ by deleting $x,y$ and all their incident
edges and joining, for each $i\in\Delta_n-\{i_1,\ldots ,i_k\}$,
the vertex $i$-adjacent to $x$ to the vertex $i$-adjacent to $y$;
$(\G^\prime,\g^\prime )$ is said to be obtained from $(\G,\g )$ by
\textit{deleting the $k$-dipole} $\theta$. Conversely $(\G,\g )$
is said to be obtained from $(\G^\prime,\g^\prime)$ by
\textit{adding the $k$-dipole}.

From now on, we restrict ourselves to $3$-manifolds; in this
context, we can introduce further moves.

Let $(\G,\g)$ be a $4$--coloured graph. Let $\Theta$ be a subgraph
of $\G$ formed by a $\{i,j\}$-coloured cycle $C$ of length $m+1$
and a $\{h,k\}$-coloured cycle $C^\prime$ of length $n+1$, having
only one common vertex $x_0$ and such that $\{i,j,h,k\} =
\{0,1,2,3\}$. Then $\Theta$ is called an \textit{(m,n)--dipole}.

If $x_1,\,x_m,\,y_1,\,y_n$ are the vertices respectively
$i,\,j,\,h,\,k$-adjacent to $x_0$, we define the $4$-coloured
graph $(\G^\prime,\g^\prime)$
 \textit{obtained from} $\G$ \textit{by cancelling the
(m,n)--dipole}, in the following way:
\begin{itemize}
\item [1)] delete $\Theta$ from $\G$ and consider the
product $\Xi$ of the subgraphs $C-\{x_0\}$ and $C^\prime-\{x_0\}$;
\par\noindent
\item [2)] for each $s,s^\prime\in \{1, \dots, n\}$
(resp. for each $r,r^\prime\in \{1, \dots, m\}$), let $e$ be the
edge joining $y_s$ and $y_{s^\prime}$ (resp. $x_r$ and
$x_{r^\prime}$) in $\G$. If $\g (e) = c \in \{0,1,2,3\}$, then,
for each $t\in \{1, \dots, m\}$ (resp. for each $t\in \{1, \dots,
n\}$), join the vertices $(x_t, y_s)$ and $(x_t, y_{s^\prime})$
(resp. $(x_r, y_t)$ and $(x_{r^\prime}, y_t)$) by a $c$-coloured
edge in $\Xi$;
\par\noindent
\item [3)] for all $r\in\{1, \dots, m\},\,s\in\{1, \dots, n\}$, if a vertex $z$
of $\Gamma -\Theta$ is joined to $y_s$ (resp. $x_r$) by a $i$ or
$j$ (resp. $h$ or $k$)--coloured edge in $\G$, then $z$ is joined
to $(x_1,\,y_s),\,(x_m,\,y_s)$ (resp. $(x_r,\,y_1),\,
(x_r,\,y_n)$) by a $i$ or $j$ (resp. $h$ or $k$)--coloured edge in
$\G'$.
\end{itemize}

Cancellation or addition of a $(m,n)$-dipole is called a
\textit{generalized dipole move}.

If two $i$-coloured edges $e,f\in E(\Gamma)$ belong to the same
$\{i,j\}$-coloured cycle and to the same $\{i,k\}$-coloured cycle
of $\Gamma$, with  $j,k\in \Delta_3-\{i\}$ (resp. to the same
$\{i,h\}$-coloured cycle of $\Gamma$, for each $h\in
\Delta_3-\{i\}$), then $(e,f)$ is called a {\it $\rho_2$-pair}
(resp. a {\it $\rho_3$-pair}). Usually, we will write \textit{
$\rho$-pair} instead of $\rho_2$-pair or $\rho_3$-pair.

A graph $\G$ is a {\it rigid crystallization} of a 3-manifold
$M^3$ if it is a crystallization of $M^3$ and contains no
$\rho$-pairs. A non-rigid crystallization $\G$ of a 3-manifold $M$
can be always transformed into a rigid one by \textit{switching
$\rho$-pairs} (see \cite{L2}) and cancelling the dipoles which
might be created in the process.

The effects of cancellation/addition of a dipole, a generalized
dipole and of switching of a $\rho$-pair are described in the
following Proposition.

\begin{proposition}  \ \label{moves_vari}{\rm (\cite{FG}, \cite{L2})}
\begin{itemize}
\item [(i)]
If $\G$ and  $\G^{\prime}$ are 4-coloured graphs representing two
3-manifolds $M$ and $M^\prime $ respectively, and $\G^{\prime}$ is
obtained from $\G$ by a dipole move or a generalized dipole move
or by switching a $\rho_2$-pair, then $M\cong M^\prime$.
\item [(ii)] Let $\G$ be a  4-coloured
graph representing a 3-manifold $M$, containing a $\rho_3$-pair.
If $\G^\prime$, obtained from $\G$ by switching it, is connected,
then it represents a 3-manifold $M^\prime$, such that $M =
M^\prime \# H$, where either $H=\mathbb S^1\times\mathbb S^2$ or
$H=\mathbb S^1\tilde\times\mathbb S^2$.
\end{itemize}
\end{proposition}

Note that, if $\G$ is a crystallization of $M$, then  $\G^\prime$
is always connected.

\begin{remark}\emph{In the case of dipole moves, statement (i) of the above Proposition is
actually stronger. In fact the main theorem of \cite{FG} proves
that $M$ ad $M^\prime$ are homeomorphic iff $\G$ and $\G^\prime$
are obtained from each other by a sequence of dipole moves.}
\end{remark}

\noindent A different set of moves is defined in \cite{LM}

\begin{remark}\emph{Each closed connected 3-manifold admits a rigid crystallization (see
\cite{CA1} for a detailed proof). Moreover, an easy consequence of
Proposition \ref{moves_vari} proves that every closed connected
3-manifold $M$ different from $M^\prime\# H$ admits a rigid
crystallization of minimal order (see \cite{CA1}). Hence, since we
are interested mainly in prime manifolds, in the generation and
analysis of our catalogues we will restrict ourselves to rigid
crystallizations.}
\end{remark}

Note that, each orientable genus two 3-manifolds is the 2-fold
covering of $\mathbb S^3$, branched over a knot or link $L$
(\cite{BH}). The construction described in \cite{F} allows to
obtain a crystallization $\Gamma$ of  $M$, starting from a
$3$-bridge presentation of  $L$. As a consequence of the
construction, $\Gamma$ belongs to a particular class of
crystallizations, called {\it 2-symmetric} in \cite{CGr}, which
can be codified by suitable 6-tuples (called {\it admissible}) of
non-negative integers (\cite{CA4}). Hence, admissible 6-tuples are
a representation tool for orientable genus two 3-manifolds.

In  \cite{GMN}, the authors describe an equivalence relation on
the set of admissible 6-tuples, whose equivalence classes consist
only of 6-tuples representing 2-symmetric crystallizations of the
same manifold. In \cite{K} a catalogue was presented of the
representatives of the equivalence classes of admissible 6-tuples,
whose associated 2-symmetric crystallizations have at most 42
vertices. Subsequently in \cite{KMN} the catalogue was reduced to
6-tuples all representing distinct manifolds.

\section{Seifert manifolds and coloured triangulations}\label{Burton-Seifert}

Let $ M =(S, (\alpha_1, \beta_1), \ldots, (\alpha_n, \beta_n))$ be
the Seifert fibered space whose orbit space is the surface $S$ and
having $n$ exceptional fibers, with non-normalized parameters
$(\alpha_i, \beta_i),\ i = 1, \ldots, n.$

Let us consider a triple of integers $(p,q,r)$ such that:

\begin{itemize}
\item [(i)] $(|p|,|q|) = 1$;
\item [(ii)] $p+q+r=0$;
\end{itemize}

In \cite{B1} and \cite{B2}, the author describes a triangulation
of the solid torus, called {\it layered solid torus of type
(p,q,r)} and denoted by $LST(p,q,r)$, which is used to construct
triangulations of Seifert manifolds.

We recall briefly the main steps of the construction, which is
done recursively. First of all, we point out that, at each step,
we obtain a layered solid torus (from now on, LST for short)
having exactly two $2$-simplexes on the boundary, which will be
called its {\it boundary faces}. Moreover, it has exactly three
boundary edges with a labelling by means of integers satisfying
conditions $(i)-(ii)$. The labelling is defined by recursion, too.

The main point of the procedure is the possibility of performing a
\textit{layering} on an edge $e^\prime$ of a layered solid torus
$LST(p,q,r)$ in the following way:

\noindent suppose that $e^\prime$ is labelled $i\in\{p,q,r\}$,
then a new tetrahedron is considered and two adjacent faces of it,
say $F$ and $F^\prime$, are identified with the boundary faces of
$LST(p,q,r)$ so that the common edge of $F$ and $F^\prime$
coincides with the edge $e^\prime$.  Let $f$ and $g$ be the
boundary edges of $LST(p,q,r)$ labelled $j$ and $k$ respectively
($j,k\in\{p,q,r\}-\{i\}$). Then the new boundary edge identified
with $f$ (resp. $g$) inherits the label $j$ (resp. $-k$). The
obtained complex is the layered solid torus whose set of related
integers is $\{j,-k,k-j\}$.

\begin{remark}\emph{The integers $p,q$ are actually the intersection numbers of the
boundary of a meridinal disk of  $LST(p,q,r)$, which is a simple
oriented curve on the boundary torus $T$ of $LST(p,q,r)$, with the
 basis of $\pi_1(T)$ formed by the (suitably oriented) edges labelled $p$ and $q$.
It is not difficult to see that $LST(-p,-q,-r)$ is the same
triangulation of the solid torus and describes the same curves
with reversed orientations. Since for our aims we don't need to
distinguish the two layered solid tori, from now on we suppose
that two of the elements of the set $\{p,q,r\}$ are positive.
Moreover, each permutation of the set $\{p,q,r\}$ doesn't change
the related LST; therefore, in the following construction we need
only to specify the set of integers we are working on, without
imposing an ordering.}
\end{remark}

In order to construct the LST with set of parameters
$\{p,q,-p-q\}$, we perform the following procedure:

\begin{itemize}
\item [-] the initial step is the LST whose parameters are
$\{1,2,-3\}$: it is obtained from the tetrahedron, with labelled
edges, in Figure 1, by identifying the ``back" faces 
according to the arrows.

\begin{figure}[h]
\begin{center}
\scalebox{0.3}{\includegraphics{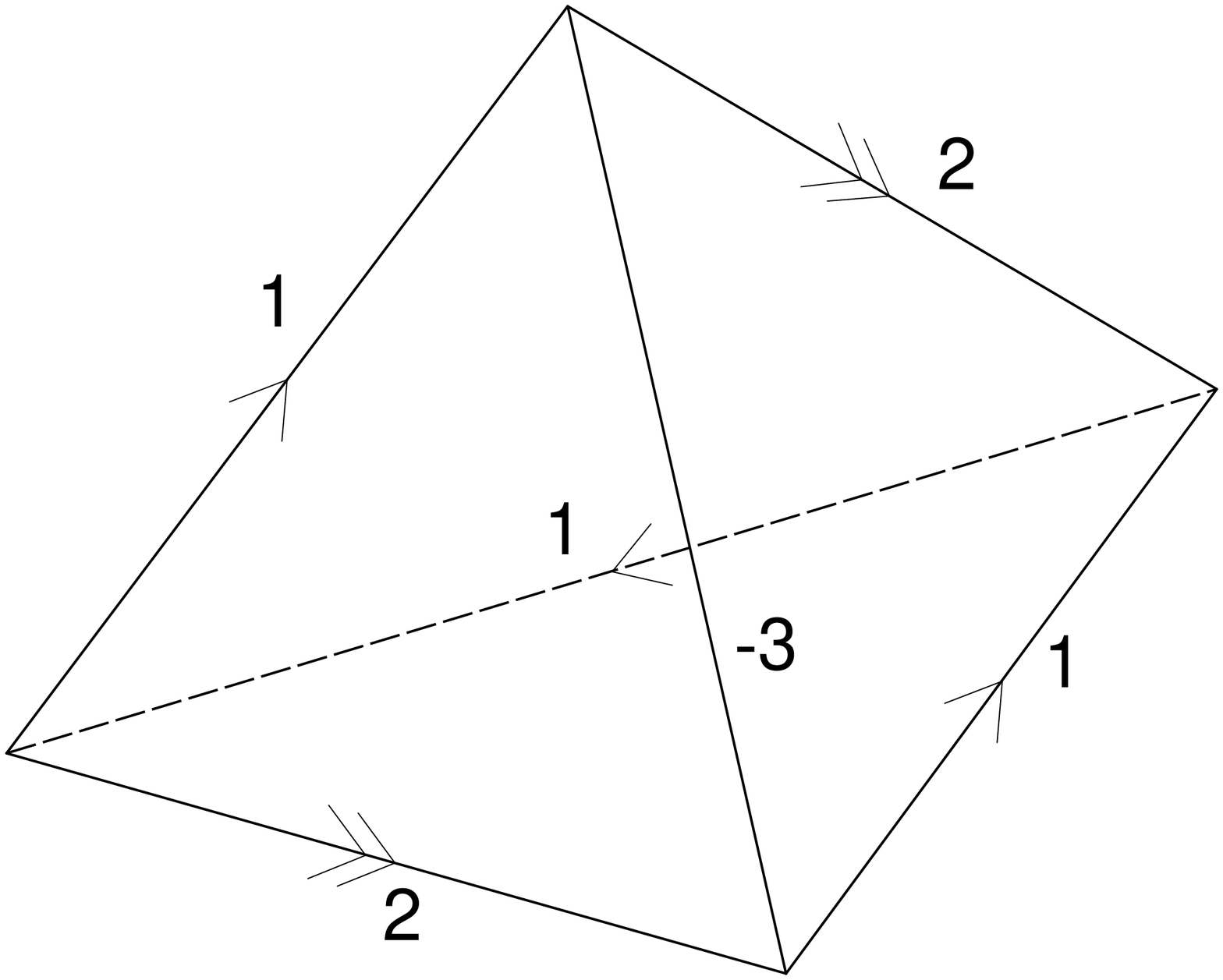}}
\end{center}
\caption{}
\end{figure}

\bigskip

\item [-] the LST with parameters $\{i,j,k\}$ ($0<j<k$) is
obtained from the LST with parameters $\{j,-k,k-j\}$, by a
layering on the edge labelled $k-j$.
\end{itemize}

For more details about the construction and its geometric
meanings, see \cite{B1} and \cite{B2}.

We are now ready to describe how to construct a coloured
triangulation of the Seifert fiber space
 $M=(\mathbb S^2,(\alpha_1,\beta_1),(\alpha_2,\beta_2),(\alpha_3,\beta_3))$.

Let us fix a point $A\in\mathbb S^2$ and let $\mathbb D_1,\
\mathbb D_2,\ \mathbb D_3$ be 2-disks in $\mathbb S^2$ such that
$\mathbb D_1\cap\mathbb D_2\cap\mathbb D_3=\partial\mathbb
D_1\cap\partial\mathbb D_2\cap\partial\mathbb D_3=\{A\}$. The
pseudocomplex $P$ of Figure 2 is a planar realization of $\mathbb
S^2\setminus\bigcup_{i=1}^3 int\,{D_i}$.

\begin{figure}[h]
\begin{center}
\scalebox{0.3}{\includegraphics{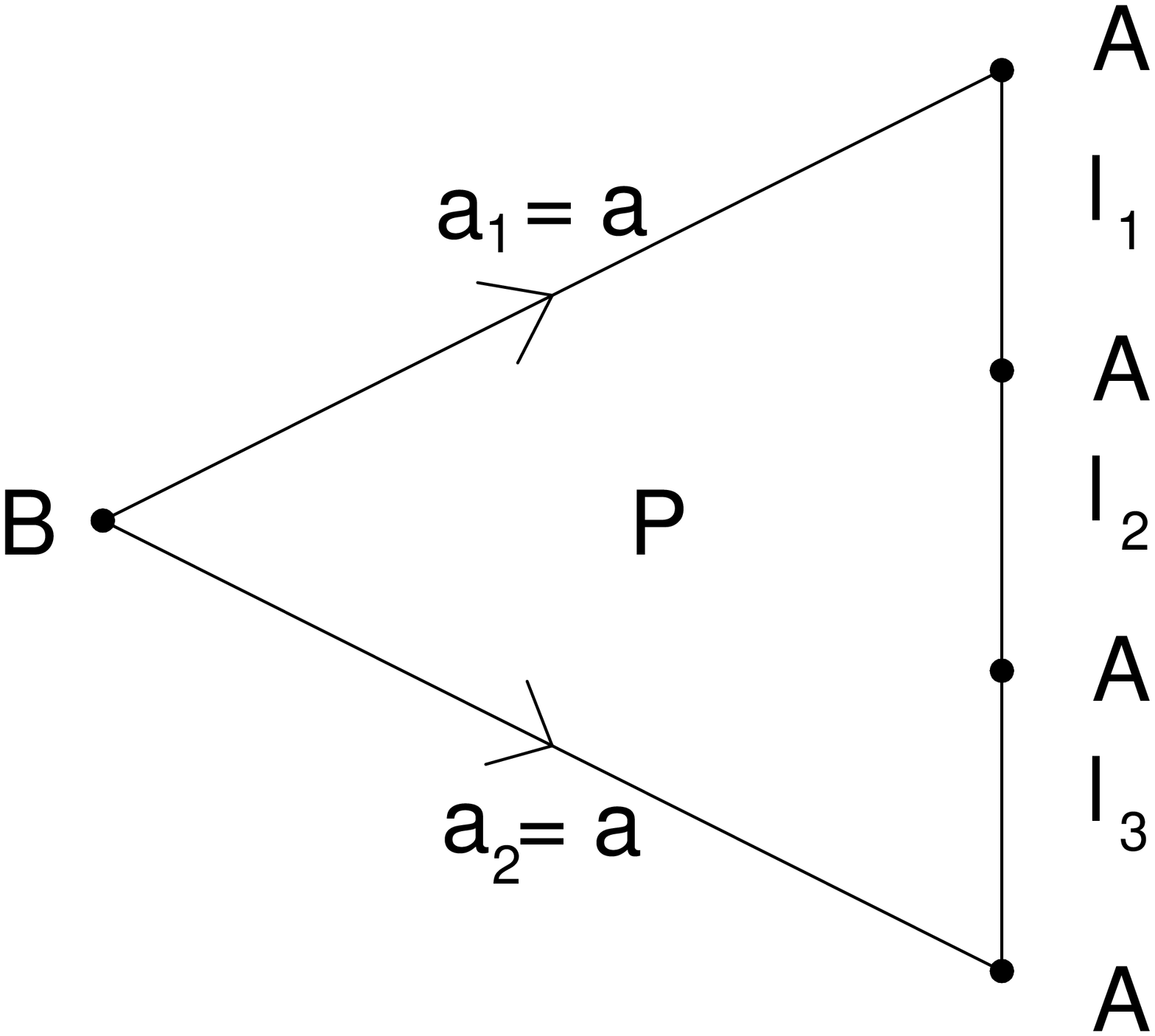}}
\end{center}
\caption{}
\end{figure}

Figure 3 shows the boundary surface of $P\times I$, with the
identifications of the faces $P\times\{0\}$ and $\{a_1\}\times I$
with $P\times\{1\}$ and $\{a_2\}\times I$ respectively. Moreover,
we marked the subdivision of the faces $\{l_i\}\times I$
($i=1,2,3$) which will be necessary for the construction of a
triangulation of $M$.

We point out that for each $i=1,2,3$, the boundary face
$\{l_i\}\times I$ is the usual representation of a torus by a
square. Therefore, after the identifications of $P\times\{0\}$ and
$\{a_1\}\times I$ with $P\times\{1\}$ and $\{a_2\}\times I$
respectively, we obtain $\mathbb S^1\times\mathbb S^2$ with three
solid tori removed.

The last step is the identification, for each $i=1,2,3$, of the
two triangles of $\{l_i\}\times I$ (see Figure 3) with the
boundary faces of $LST(\alpha_i,\theta_i,\sigma_i)$, where either
$\theta_i = \beta_i$ or $\sigma_i = - \beta_i$, so as to identify
the edges $\{A\}\times I$ (resp. $l_i\times\{0\}$ and
$l_i\times\{1\}$) with the edges labelled $\alpha_i$ (resp. either
$\theta_i$ or $\sigma_i$). In this way we obtain a triangulation
$K$ of $M$.

Note that, in order to define precisely the gluing of the three
layered solid tori, it is necessary to specify not only the sets
of parameters but also their ordering. Hence the notation with
triples $(\alpha_i,\theta_i,\sigma_i),\ \ i=1,2,3$.

\begin{figure}[h]
\begin{center}
\scalebox{0.6}{\includegraphics{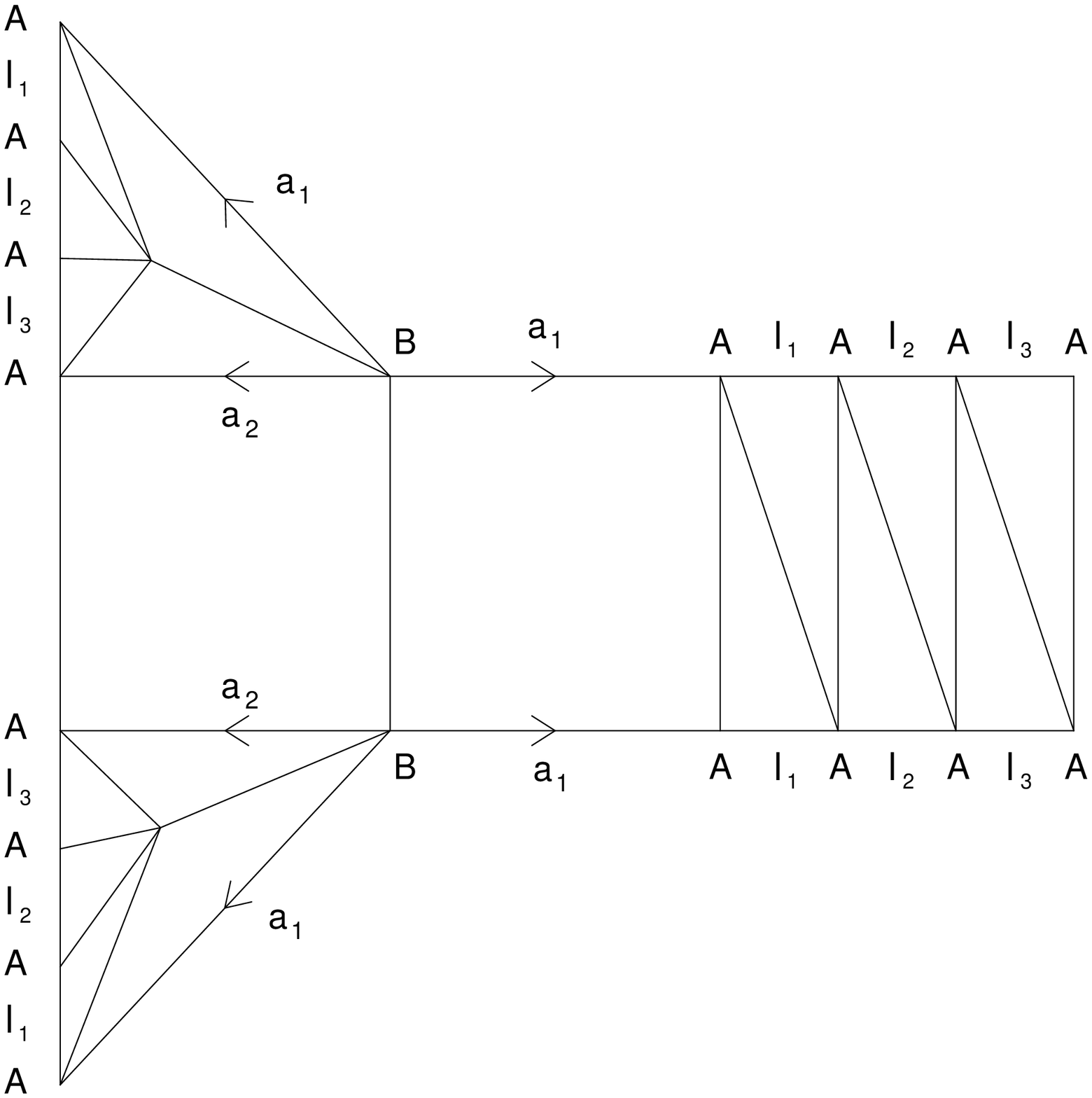}}
\end{center}
\caption{}
\end{figure}

Finally, in order to obtain a coloured triangulation, we consider
the first barycentric subdivision $K^\prime$ of $K$ and colour $h$
($h\in\Delta_3$) the barycenters of $h$-dimensional cells of
$K^\prime$.

\section{Generating and analysing genus two crystallizations}

The \textit{regular genus} of an $n$-manifold is a combinatorial
invariant which extends to dimension $n$ the classical concepts of
genus of a surface and Heegaard genus of a 3-manifold. For the
precise definition of the invariant, we refer to \cite{BCG}, but
to suit our present aim it is sufficient to recall that, if $\G$
is a crystallization of a 3-manifold, then the regular genus of
$\G$ is the integer
$$\rho(\G)=min\{g_{01}(\G), g_{02}(\G), g_{03}(\G)\}-1$$
\noindent where for each $i=1,2,3,\ \ g_{0i}(\G)$ is the number of
$\{0,i\}$-residues of $\G$.

Therefore, the regular genus of a 3-manifold $M$ can be defined as
the minimum $\rho(\G)$ among all crystallizations of $M$.

Combinatorial encoding of closed 3-manifolds by crystallizations,
together with the restriction to rigid ones, allows us to
construct essential catalogues of all contracted triangulations of
closed 3-manifolds up to a certain number of vertices. By using
the codes, we can easily avoid isomorphic graphs, too.

The general algorithm, which is fully described in \cite{CA1} and
\cite{CC}, runs as follows:

\begin{itemize}
\item [--] The starting point is the set $\mathcal S^{(2p)}$ of all (connected)
rigid\footnote{here ``rigid" means without $\rho_2$-pairs} and
planar 3-coloured graphs with $2p$ vertices. The construction
makes use of Lins's results in \cite{L1} and is performed by
induction on $p$. More precisely every rigid and planar 3-coloured
graph with $2p$ vertices is obtained from an analogous one with
$2p-2$ vertices, by a suitable operation (\textit{antifusion}),
with the possible exception of the ''prism'' with $p$-gonal base
(with $p$ even);

\item [--] to each element of $\mathcal S^{(2p)}$ 3-coloured edges are added in all possible ways so as to produce
rigid crystallizations of 3-manifolds. By checking bipartition,
crystallizations of orientable and non-orientable manifolds can be
separated.
\end{itemize}

In this way we obtain the catalogues of all non-isomorphic rigid
bipartite and non-bipartite crystallizations with $2p$ vertices.

It is possible to modify the general algorithm in order to
construct only crystallizations having a fixed regular genus. In
particular, we modified Lins's construction in order to obtain the
set $\tilde{\mathcal S}^{(2p)}$ of all rigid 3-coloured graphs
$\tilde{\Sigma}$  with $2p$ vertices representing $\mathbb S^2$
and such that $g_{ij}(\tilde{\Sigma})=3$, for at least one pair of
distinct colours $i,j\in\{0,1,2\}$ and $g_{hk}(\tilde{\Sigma})\geq
3$ for the remaining pairs. Since Lins's procedure is inductive on
the number of vertices and each step increases the number of
bicoloured cycles of a given pair of colours by at most one, it
was sufficient, at each step, to perform the required
transformations only on graphs having at most two
$\{i,j\}$-residues for at least one pair $i,j\in\{0,1,2\}$ and
finally to eliminate the resulting graphs with $2p$ vertices which
satisfy the same property.

By adding $p$ 3-coloured edges to each element of $\tilde{\mathcal
S}^{(2p)}$ in all possible ways so as not to destroy planarity of
the $\hat j$-residues ($j\in\{0,1,2\}$) and rigidity, and by
eliminating the resulting graphs $\G$ which are not contracted, we
obtained the catalogues ${\mathcal C}^{(2p)}_2$ and $\tilde
{\mathcal C}^{(2p)}_2$ of all rigid bipartite and non-bipartite,
respectively, non-isomorphic crystallizations with $2p$ vertices
having regular genus two.

The restriction about the genus allows us to obtain a reduction of
time in the generation procedure and, consequently, we could
obtain catalogues for higher number of vertices than in the
general case where no genus bounds are imposed. More precisely, we
generated these catalogues up to 42 vertices.

The above algorithm was implemented in a C++ program, whose output
data are presented in Table 1 according to the number of vertices.
We point out that there are no rigid genus two crystallizations
with less than 14 vertices.

\bigskip

\centerline{\begin{tabular}{|c|| c | c | c | c | c | c | c | c | c
| c | c | c | c | c | c|}
 \hline
 {\bf 2p } &  14 & 16 & 18 & 20 & 22 & 24 & 26 & 28 & 30 & 32 & 34 & 36 & 38 & 40 & 42\\
  \hline  \hline  \ & \ & \ & \ & \ & \ & \ & \ & \ & \ & \ & \ & \ & \ & \ & \ \\
 \hfill {\bf  $\#\mathcal C^{(2p)}_2$} \hfill & 1 & 2 & 4 & 6 & 8 & 14 & 18 & 23 & 38 & 47 & 58 & 79 & 118 & 128 & 159\\
 \hline
 \hfill \ & \ & \ & \ & \ & \ & \ & \ & \ & \ & \ & \ & \ & \ & \
 &  \ \\
 {\bf $\#\tilde {\mathcal C}^{(2p)}_2$} \hfill & 1 & 1 & 1 & 1 & 2 & 2 & 3 & 2 & 6 & 7 & 9 & 7 & 12 & 12 & 16\\
 \hline
 \end{tabular}}

\bigskip
\centerline{\textbf{Table 1:  Genus two rigid crystallizations up
to 42 vertices.}}

In \cite{CC} for the orientable case and \cite{BCrG} and
\cite{CA1} for the non-orientable case, catalogues of all rigid
crystallizations with at most 30 vertices have been analysed and
the represented manifolds identified.

In this paper we follow the same line with respect to the
catalogues ${\mathcal C}^{(2p)}_2$ and $\tilde {\mathcal
C}^{(2p)}_2$ with $p\leq 21$, i.e. we manipulate crystallizations
through generalized dipole moves and subdivide them into classes
according to the equivalence defined by the moves.

In fact, the procedure is completely general and requires only a
given list $X$ of rigid crystallizations and a fixed set $\mathcal
S$ of sequences (called \textit{admissible}) of generalized
dipoles moves, dipole moves and $\rho$-pairs switching, such that
each element of $\mathcal S$ transforms rigid crystallizations
into rigid crystallizations (see \cite{CC} for details). For each
$\Gamma\in X$ and for each $\epsilon\in\mathcal S$, we denote by
$\theta_\epsilon (\G)$ the (rigid) crystallization obtained by
applying the sequence $\epsilon$ to $\G$.

Note that, by Proposition \ref{moves_vari}, if $\G$ represents a
3-manifold $M$ then $\theta_\epsilon (\G)$ re\-pre\-sents the
3-manifold $M^\prime$, such that $M=M^\prime \#_r H$ where $\#_r
H$ denotes the connected sum of $r$ copies either of the
orientable or of the non-orientable $\mathbb S^2$-bundle over
$\mathbb S^1$; more precisely $H=\mathbb S^1\times\mathbb S^2$ iff
$\G$ and $\theta_\epsilon (\G)$ are both bipartite or both
non-bipartite. Obviously $r$ is the number of $\rho_3$-pairs,
which have been eventually deleted while applying the sequence
$\epsilon$ (usually we denote this number by $h_\epsilon (\G)$).

As a consequence, by using the elements of $\mathcal S$, it is
possible to subdivide $X$ into disjoint classes $\{c_1,\ldots
,c_s\}$ such that, for each $i\in\{1,\ldots ,s\}$ and for each $\G
,\G^\prime\in c_i$, there exist two integers $h,k\geq 0$ and a
3-manifold $M$ such that $|K(\G)|=M\#_h H$ and
$|K(\G^\prime)|=M\#_k H$.

More precisely, for each $\G\in X$, we define the \textit{class}
of $\G$ as the set
\begin{align*}
cl(\G )=\{&\G^\prime\in X\ |\ \exists\,\epsilon,\epsilon^\prime\in
\mathcal S\ \text{\ s.t.\ }\\
&\theta_\epsilon(\G ) \text{\ and\ }
\theta_{\epsilon^\prime}(\G^\prime) \text{\ have the same code}\}.
\end{align*}
In \cite{CC} and \cite{BCrG} it is shown how to construct the set
$cl(\G)$, for each $\G\in X$, and also how to compute a
non-negative number denoted by $h(\G)$; it defines a function $h\
:\ X\to\mathbb N\cup\{0\}$ inducing a natural subdivision of each
class $c_i\ (i=1,\ldots,s\ )$ into subclasses
$c_{i,k}=\{\G^\prime\in c_i\ |\ h(\G^\prime)=k\}$ such that all
crystallizations of a given subclass represent the same manifold.

Moreover, for each $\G\in X$ such that $cl(\G)=c_i$, if the
elements of $c_{i,0}$ represents a 3-manifold $M$, then $\G$
represents the 3-manifold $M^\prime=M\#_{h(\G)}H$, with $H$ as
above.

For the precise description of the algorithm yielding $cl(\G)$ and
$h(\G)$, we refer to the already cited papers. Moreover, we point
out that the choice of the set $\mathcal S$, which is used in our
implementation, is exactly the same as described in \cite{BCrG}.

Note that, by the above definitions and results, if known
catalogues of crystallizations are inserted in $X$, all classes of
$X$ containing at least one known crystallization are completely
identified, with respect to the manifolds represented by all their
subclasses.

Moreover, condition (\#) can be checked to recognize connected
sums.

According to these ideas, the classification algorithm has been
implemented in the C++ program {\it $\G$-class}\,\footnote{
developed by M.R. Casali and P. Cristofori
and 
available at WEB page\\
http://cdm.unimo.it/home/matematica/casali.mariarita/CATALOGUES.htm
where a detailed description of the program can be found, too.}:
its input data are a list $X$ of rigid crystallizations and the
informations about already known crystallizations in $X$ (possibly
none), i.e the identification of their represented manifolds
through suitable ``names"; the output is the list of classes of
$X$, together with their representatives and, if possible, their
names.

In the following sections we present the results of $\G$-class
applied to catalogues ${\bf C}^{42}_2=\bigcup_{p=1}^{21}{\mathcal
C}^{(2p)}_2$ and $\widetilde{\bf
C}^{42}_2=\bigcup_{p=1}^{21}\widetilde {\mathcal C}^{(2p)}_2$.

\section{Genus two orientable 3-manifolds}

The catalogue $\textbf C^{42}_2$ is partitioned by the program
$\G$-class into 175 classes, 93 of which are known through former
results in \cite{L2} and \cite{CC}; moreover the program
recognized 23 connected sums, which didn't appear in the cited
papers.

In particular, we have the following result.

\begin{lemma}\label{somme or} For each {\rm $(\Gamma,\gamma)\in\textbf C^{(42)}_2$}
satisfying condition (\#) with summands
$(\Gamma^{(1)},\gamma^{(1)})$ and $(\Gamma^{(2)},\gamma^{(2)})$,
then
\begin{itemize}
\item  [-] if at least one of $|K(\Gamma^{(1)})|,\ |K(\Gamma^{(2)})|$ admits
orientation-reversing self-ho\-meo\-morphisms, then $cl(\Gamma)$
is the only class representing a connected sum with summands
$|K(\Gamma^{(1)})|$ and $|K(\Gamma^{(2)})|$;
\item  [-] otherwise there are exactly two classes $\ cl(\Gamma),\
cl(\Gamma^\prime)\ $ representing a connected sum with summands
$|K(\Gamma^{(1)})|$ and $|K(\Gamma^{(2)})|$; in this case
$|K(\Gamma)|\ncong |K(\Gamma^\prime)|$.
\end{itemize}
\end{lemma}

\begin{proof}
Our program results show that, whenever at least one of
$|K(\Gamma^{(1)})|$,\ $|K(\Gamma^{(2)})|$ admits
orientation-reversing self-homeomorphisms, there is only one class
as in the statement.

On the other hand, in the case of neither $\ |K(\Gamma^{(1)})|\ $
nor $\ |K(\Gamma^{(2)})|\ $ admitting orientation-reversing
self-homeomorphisms, there are exactly two classes $cl(\Gamma ),\
cl(\Gamma^\prime)\ $ representing a connected sum with
$|K(\Gamma^{(1)})|$ and $|K(\Gamma^{(2)})|$ as summands.

Furthermore, we point out that, if $\Gamma $ is the connected sum
of the graphs $\ (\Gamma_1,\gamma_1)\ $ and $\
(\Gamma_2,\gamma_2)\ $ with respect to the vertices $v_1\in
V(\Gamma^{(1)})$ and $v_2\in V(\Gamma^{(2)})$, then, by choosing
in $\ V(\Gamma^{(2)})\ $ a vertex $v^\prime_2$ belonging to a
different bipartition class from $v_2$, we can construct a
connected sum $\Gamma^{\prime\prime}$ of $\ (\Gamma_1,\gamma_1)\ $
and $\ (\Gamma_2,\gamma_2)\ $ (with
$\#V(\Gamma^{\prime})=\#V(\Gamma^{\prime\prime})\le 42$) such that
$|K(\Gamma)|\ncong |K(\Gamma^{\prime\prime})|$. Moreover, by a
suitable choice of $v^\prime_2$, we can obtain that $\
\Gamma^{\prime\prime}\ $ has regular genus two. Actually $\
\Gamma\ $ and $\ \Gamma^{\prime\prime}\ $ represent the two
non-homeomorphic connected sums of $\ |K(\Gamma^{(1)})|\ $ and $\
|K(\Gamma^{(2)})|\ $.

Since $\Gamma^{\prime\prime}$ must belong to $\textbf C^{(42)}_2$
too, we have necessarily $\Gamma^{\prime\prime}\in
cl(\Gamma^\prime)$.\end{proof}

\begin{remark}\emph{The identification of the summands $\ |K(\Gamma^{(1)})|\ $
and $\ |K(\Gamma^{(2)})|\ $ involved in the above lemma, has been
done directly by the program for a large number of classes; namely
those having $\ \Gamma^{(1)}\ $ and $\ \Gamma^{(2)}\ $ with less
than 32 vertices. All other summands had cyclic fundamental
groups. Therefore, we constructed crystallizations of genus one of
lens spaces with the required groups and inserted them in the list
handled by $\G$-class. The program identified all unknown summands
as lens spaces in our list.}
\end{remark}

Further 29 classes of $\textbf C^{42}_2$, having cyclic
fundamental groups, were recognized by applying the same procedure
as described in the above remark. They all turned out to represent
lens spaces, which don't appear in catalogues $\mathcal C^{(2p)},\
\ 1\leq p\leq 15$ or among the manifolds of \cite{BGR}.
Furthermore, as we will see in the following, no lens space is
left among the still unidentified manifolds.

The main consequence of the output results of $\G$-class and the
above lemma, is that a bjiective correspondence exists between the
already identified subclasses and the represented manifolds. As we
will prove, the same holds for the whole catalogue.

The classes of $\textbf C^{42}_2$, which have not been identified
by $\G$-class, are 30. The pro\-blem of their recognition will be
discussed and wholly solved in the following sections.

A comparison of codes yields that there is a bijective
correspondence between our still unknown classes of
crystallizations of $\textbf C^{42}_2$ and 30 of the 78 6-tuples
which in \cite{KMN} are proved to represent all distinct prime
orientable genus two 3-manifolds admitting a coloured
triangulation with $2p$ tetrahedra, with $p\leq 21$ and having
acyclic fundamental groups. Among these classes three are already
identified by the results of \cite{BGR} up to 34 tetrahedra.

The 48 manifolds which have been already identified by $\G$-class,
by means of their admitting at least one coloured triangulation
with less than 32 tetrahedra (see \cite{CC}), are mostly Seifert
spaces with base $\mathbb S^2$ and three exceptional fibers (for
the complete list see Appendix A).

The remaining unidentified manifolds fall into two cases: those
with finite and those with infinite fundamental group. As pointed
out in \cite{KMN}, all finite groups are Milnor's groups,
therefore the correspoding manifolds are elliptic and completely
known. In Appendix A, besides explicitly writing down the groups,
we specified also the Seifert structure of these manifolds.

The pro\-blem of identifying the manifolds with infinite
fundamental group is left open by the authors of \cite{KMN}. We
are solving it by manipulating group presentations and by
constructing coloured triangulations of Seifert spaces with base
$\mathbb S^2$ and three exceptional fibres. In fact all manifolds
under examination turn out to belong to this family.

Our starting point is the following Proposition, which enables us
to recognize all groups in Karabas-Malicky-Nedela's list (in the
following ``KMN-list" for short) corresponding to our still
unknown manifolds, as fundamental groups of Seifert spaces of the
above described type.

\medskip

\begin{proposition}\label{inf_fund_group}\ \

\begin{itemize}
\item [(i)] the group $G(\alpha_1,\alpha_2,\alpha_3)$ defined by the presentation

$$<a,b\ /\ a^{\alpha_1}=b^{\alpha_2}=(ab)^{\alpha_3}>,\quad \alpha_i>0,\ \text{for each } i=1,2$$

is isomorphic to the fundamental group of the Seifert manifold

$$(\mathbb S^2,(\alpha_1,1),(\alpha_2,1),(|\alpha_3|,\varepsilon)),\ \text{\ where\ \ } \varepsilon =-\alpha_3/|\alpha_3|;$$

\item [(ii)] the group

$$G^\prime(\alpha_1,\alpha_2,\alpha_3)=<a,b\ /\ a^{\alpha_3}=b^{\alpha_2}=(a^{-\varepsilon}b^\varepsilon)^{\alpha_1}>,$$
with $\quad \alpha_i>0,\ \text{for each } i=1,2,3,\ \
\varepsilon=\pm 1$, is isomorphic to the fundamental group of the
Seifert manifold

$$(\mathbb S^2,(\alpha_1,1),(\alpha_2,-\varepsilon),(\alpha_3,\varepsilon)).$$

\item [(iii)] the group

$$G^{\prime\prime}=<a,b\ /\ a^{5}=b^{3}=(ab^{-2})^{-3}>$$

is isomorphic to the fundamental group of the Seifert manifold

$$(\mathbb S^2,(3,1),(3,1),(5,-4)).$$

\end{itemize}
\end{proposition}

\begin{proof}
\begin{itemize}
\item [(i)] Let us set $q_1=a,\ q_2=b,\ q_3=(ab)^{-1},\ h=(ab)^{-\alpha_3}$, then from the relations of $G$, we have
$$q_1^{\alpha_1}=q_2^{\alpha_2}=h^{-1},\ q_3^{|\alpha_3|}h^\varepsilon=1,\ q_1q_2q_3=1$$
Moreover, it easy to see that, for each $i=1,2,3$, $q_i$ and $h$
commute.

\noindent Therefore $G$ admits the presentation
\begin{eqnarray}\nonumber
<&q_1,\ q_2,\ q_3,\ h\ /\ q_1h=hq_1,\ q_2h=hq_2,\ q_3h=hq_3,\
q_1^{\alpha_1}h=1,\\ \nonumber &q_2^{\alpha_2}h=1,\
q_3^{|\alpha_3|}h^\varepsilon=1,\ q_1q_2q_3=1>
\end{eqnarray}
\noindent which is a well-known presentation of $\pi_1((\mathbb
S^2,(\alpha_1,1),(\alpha_2,1),(|\alpha_3|,\varepsilon)))$ (see
\cite{Or}).

\item [(ii)] Let us consider the case $\varepsilon =1$, then
$$G^\prime(\alpha_1,\alpha_2,\alpha_3)=<a,b\ /\ a^{\alpha_3}=b^{\alpha_2}=(a^{-1}b)^{\alpha_1}>$$
Set $q_1=a^{-1}b,\ q_2=b^{-1},\ q_3=a,\ h=(a^{-1}b)^{-\alpha_1}$

\noindent We have
$$q_1^{\alpha_1}=q_2^{-\alpha_2}=q_3^{\alpha_3}=h^{-1},\ q_1q_2q_3=1$$
\noindent and again each $q_i$ ($i=1,2,3$) commutes with $h$.

\noindent Therefore $G^\prime\cong\pi_1((\mathbb
S^2,(\alpha_1,1),(\alpha_2,-1),(\alpha_3,1)))$. The case
$\varepsilon =-1$ is analogous.

\item [(iii)] If we set $q_1=ab^{-2},\ q_2=b^{-1},\ q_3=a^{-1},\ h=(ab^{-2})^{-3}$, then we have the relations
$$q_1^{3}=q_2^{3}=q_3^{5}=h^{-1},\ q_2^2=q_3q_1$$
\noindent Hence $h^{-1}=q_2^3=q_3q_1q_2$. Since each $q_i$
($i=1,2,3$) commutes with $h$, we can write
$$q_3^{-1}h^{-1}=h^{-1}q_3^{-1}\ \ \Longrightarrow\ \ q_1q_2=q_3q_1q_2q_3^{-1}\ \ \Longrightarrow\ \ q_1q_2q_3=q_3q_2q_1$$

\noindent By comparing the relations we have
$$h^{-1}=q_3q_1q_2=q_1q_2q_3$$
\noindent Therefore
\begin{eqnarray}\nonumber
G^{\prime\prime}\cong<&q_1,\ q_2,\ q_3,\ h\ /\ q_1h=hq_1,\
q_2h=hq_2,\ q_3h=hq_3,\ q_1^{3}h=1,\\ \nonumber &q_2^{3}h=1,\
q_3^{5}h=1,\ q_1q_2q_3=h^{-1}>
\end{eqnarray}
\noindent which is the fundamental group of
$$(\mathbb
S^2,(3,1),(3,1),(5,1),(1,-1))=(\mathbb S^2,(3,1),(3,1),(5,-4)).$$

\end{itemize}
\end{proof}

We point out that all unknown 6-tuples in KMN-list corresponding
to manifolds which were not identified by our former results have
fundamental group admitting a presentation of one of the above
types.

For each of these 6-tuples $f$, by means of the algorithm
described in section \ref{Burton-Seifert}, we constructed a
coloured triangulation $\Gamma(f)$ of the Seifert manifold
$M=(\mathbb
S^2,(\alpha_1,\beta_1),(\alpha_2,\beta_2),(\alpha_3,\beta_3)),$
with parameters $\alpha_i,\beta_i$ ($i=1,2,3$) determined by the
presentation of the fundamental group of $\Gamma(f)$ given in
KMN-list and by Proposition \ref{inf_fund_group}. In the following
table, we present for each of these Seifert manifolds, the triples
of parameters $(\alpha_i,\theta_i,\sigma_i)$ ($i=1,2,3$) of the
three layered solid tori which have been glued to the three
boundary components of $P\times I$, in order to obtain the
required triangulation of $M$.

As a consequence, by cancellation of dipoles and switching of
$\rho$-pairs in the above coloured triangulations, we obtain
easily a list $Y$ of crystallizations of the Seifert manifolds
which could match our unknown classes.

\bigskip
\centerline{\begin{tabular}{|c|c|}
  \hline\ & \  \\
  \hfill  \textbf{Seifert manifold} \hfill &
\textbf{Layered Solid tori}  \\
\hline\ & \\
      $(\mathbb S^2,(3,1),(3,2),(4,-3))$ & $(3,1,-4), (3,2,-5), (4,-7,3)$\\
      $(\mathbb S^2,(2,1),(4,1),(4,-1))$ & $(2,1,-3), (4,1,-5), (4,-5,1)$\\
      $(\mathbb S^2,(2,1),(4,1),(5,-4))$ & $(2,1,-3), (4,1,-5), (5,-9,4)$\\
      $(\mathbb S^2,(3,1),(3,1),(3,1))$ & $(3,-2,-1), (3,1,-4), (3,1,-4)$\\
      $(\mathbb S^2,(3,1),(3,1),(4,-1))$ & $(3,1,-4), (3,1,-4), (4,-5,1)$\\
      $(\mathbb S^2,(2,1),(3,1),(7,-6))$ & $(2,1,-3), (3,1,-4), (7,-13,6)$\\
      $(\mathbb S^2,(3,1),(3,1),(4,-3))$ & $(3,1,-4), (3,1,-4), (4,-7,3)$\\
      $(\mathbb S^2,(2,1),(3,2),(6,-5))$ & $(2,1,-3), (3,2,-5), (6,-11,5)$\\
      $(\mathbb S^2,(3,1),(3,2),(5,-4))$ & $(3,1,-4), (3,2,-5), (5,-9,4)$\\
      $(\mathbb S^2,(2,1),(3,1),(6,-1))$ & $(2,1,-3), (3,1,-4), (6,-7,1)$\\
      $(\mathbb S^2,(2,1),(4,1),(5,-3))$ & $(2,1,-3), (4,1,-5), (5,-8,3)$\\
      $(\mathbb S^2,(2,1),(4,3),(5,-4))$ & $(2,1,-3), (4,3,-7), (5,-9,4)$\\
      $(\mathbb S^2,(2,1),(4,1),(6,-5))$ & $(2,1,-3), (4,1,-5), (6,-11,5)$\\
      $(\mathbb S^2,(3,1),(3,2),(3,-1))$ & $(3,1,-4), (3,2,-5), (3,-4,1)$\\
      $(\mathbb S^2,(2,1),(4,1),(4,1))$ & $(2,1,-3), (4,-3,1), (4,1,-5)$\\
      $(\mathbb S^2,(3,2),(4,1),(4,-3))$ & $(3,2,-5), (4,1,-5), (4,-7,3)$\\
      $(\mathbb S^2,(2,1),(4,1),(5,-1))$ & $(2,1,-3), (4,1,-5), (5,-6,1)$\\
      $(\mathbb S^2,(2,1),(5,1),(5,-4))$ & $(2,1,-3), (5,1,-6), (5,-9,4)$\\
      $(\mathbb S^2,(3,1),(3,1),(4,1))$ & $(3,1,-4), (3,1,-4), (4,-3,-1)$\\
      $(\mathbb S^2,(3,1),(4,1),(4,-1))$ & $(3,1,-4), (4,1,-5), (4,-5,1)$\\
      $(\mathbb S^2,(3,1),(3,1),(5,-1))$ & $(3,1,-4), (3,1,-4), (5,-6,1)$\\
      $(\mathbb S^2,(2,1),(3,1),(8,-7))$ & $(2,1,-3), (3,1,-4), (8,-15,7)$\\
      $(\mathbb S^2,(2,1),(3,1),(7,-5))$ & $(2,1,-3), (3,1,-4), (7,-12,5)$\\
      $(\mathbb S^2,(3,1),(3,1),(5,-4))$ & $(3,1,-4), (3,1,-4), (5,-9,4)$\\
\hline \end{tabular}}

\bigskip

The following Proposition and its  corollary solve the recognition
pro\-blem both for the crystallizations of $\textbf C^{42}_2$ and
for the 6-tuples of \cite{KMN}.

\begin{proposition}
There are exactly 78 genus two prime orientable 3-manifolds
admitting a coloured triangulation with at most 42 tetrahedra and
regular genus two. They are:
\begin{itemize}
\item [-] seventy-three Seifert manifolds\;\footnote{thirty-nine elliptic, four flat, ten with Nil and twenty with $\widetilde{SL}_2(\mathbb R)$ geometry};
\item [-] three Dehn-fillings (of the complement of link $6^3_1$)\;\footnote{two of these manifolds admit also a torus bundle structure with Sol geometry, the remaining one is hyperbolic};
\item [-] two non-geometric 
graph-manifolds;
\end{itemize}
\end{proposition}

\begin{proof}
48 manifolds appeared already in catalogues $\mathcal C^{(2p)}$
with $p\leq 15$ and program $\G$-class proved that at least one of
their genus two crystallizations is equivalent (by dipole and
generalized dipole moves and $\rho$-pair switchings) to a
crystallization with less than 32 vertices. Among them there are
the three Dehn-fillings and the two non-geometric graph-manifolds.
Further three manifolds are listed in \cite{BGR}.
Of the remaining ones, 11 admit finite fundamental group 
and could be recognized through Milnor's list of groups. In all
cases, by the results in \cite{KMN}, the group identifies
univocally the manifold.

We remark once more that the remaining 16 classes represent
manifolds with infinite fundamental groups of the type (i), (ii)
or (iii) in Proposition \ref{inf_fund_group} (see Appendix A of
\cite{KMN}). Therefore, we added to the set $Y$, which we
described above, the crystallizations of the unknown classes and
we applied program $\G$-class to the resulting list. The output
results proved that the suspected identifications were true.
\end{proof}

Table 2 of Appendix A contains KMN-list of 6-tuples together with
their re\-pre\-sen\-ted manifolds according to the results
summarized in the above Proposition.

\section{Genus two non-orientable 3-manifolds}

$\G$-class, applied to $\widetilde {\textbf C}^{(42)}_2$, produced
nine classes, which were all recognized by the program by means of
the inserted catalogues $\widetilde {\mathcal C}^{(2p)}$ with
$1\leq p\leq 15$ (see (resp. \cite{BCrG})) and correspond to nine
distinct manifolds, including  $\mathbb S^1\tilde\times\mathbb
S^2$ (of genus one) and the connected sum $L(2,1) \# (\mathbb
S^1\tilde\times\mathbb S^2)$  . As a consequence we can state the
following Proposition\footnote{for notations see Appendix A}.

\begin{proposition} There exist exactly seven non-orientable prime
genus two 3-manifolds admitting a coloured triangulation with at
most 42 tetrahedra and regular genus two. They are

\begin{itemize}
 \item $\mathbb {RP}^2\times\mathbb S^1$
  \item the two flat manifolds $\mathbb E^3/Bb$ and $\mathbb
  E^3/Pna2_1$
  \item the three torus bundles $TB\begin{pmatrix} 0 & 1\\ 1 & -1\end{pmatrix}$,
  $TB\begin{pmatrix} 2 & 1\\ 1 & 0\end{pmatrix}$ and
  $TB\begin{pmatrix} 3 & 1\\ 1 & 0\end{pmatrix}$ with Sol geometry;
  \item the Seifert manifold $(\mathbb R \mathbb
  P^2;(2,1),(3,1))$ with geometry $\ \mathbb H^2\times\mathbb R$.
\end{itemize}

All these manifolds, excepting $\mathbb {RP}^2\times\mathbb S^1$,
also admit coloured triangulations of strictly higher genus with
30 or less tetrahedra.
\end{proposition}

More precisely, in Table 3 of Appendix A we present the list of
the above manifolds according to the number of vertices of their
minimal genus two crystallization.

\appendix
\section{Appendix A}

Table 2 (resp. Table 3) presents the catalogue of Heegaard genus
two prime orientable (resp. non-orientable) 3-manifolds admitting
a crystallization (with regular genus two)
 with at most 42 vertices. The manifolds are
identified via their JSJ decomposition or fibering structure and,
possibly, via a further structure as quotient of $\mathbb S^3$ or
$\mathbb E^3$. The second and last column of Table 2 contain the
informations about the 6-tuple which represent the manifold in
KMN-list and its position in the same list.

As far as the identification of a manifold is concerned, the
following notations are used:

\begin{itemize}
\item[-]  $\mathbb S^3/G$ is the quotient space of $\mathbb S^3$
by the action of the group $G$; the involved groups are groups of
type
\begin{itemize}
\item [\ ] $Q_{4n}=<x,y\, |\, x^2=(xy)^2=y^n>$, ($n\geq 2$)
\item [\ ] $D_{2^k(2n+1)}=<x,y\, |\, x^{2^k}=1, y^{2n+1}=1, xyx^{-1}=y^{-1}>$, \ ($k\geq 3,\ n\geq 1$),
\item [\ ] $P_{24}=<x,y\, |\, x^2=(xy)^3=y^3,\ x^4=1>$,
\item [\ ] $P_{48}=<x,y\, |\, x^2=(xy)^3=y^4,\ x^4=1>$,
\item [\ ] $P_{120}=<x,y\, |\, x^2=(xy)^3=y^5,\ x^4=1>$,
\item [\ ] $P^\prime_{3^k8}=<x,y,z\, |\, x^2=(xy)^2=y^2,\ zxz^{-1}=y,\ zyz^{-1}=xy,\ z^{3^k}=1>$, ($k\geq 2$)
\end{itemize}
 or direct products of the above with cyclic groups $\mathbb Z_n$ ($n\in \mathbb Z^+$);
\item[-] $\mathbb E^3/G$ is the quotient space of $\mathbb E^3$
by the action of the group $G$; the notations for groups $G$ are
those of the International Tables for Crystallography (see also
\cite{V1} and \cite{V2}, where the alternative notations, used in
\cite{BCrG} and  \cite{CA1}, were introduced, too).
\item[-] as in section \ref{Burton-Seifert}, $(S, (\alpha_1, \beta_1), \ldots, (\alpha_n,
\beta_n))$ is the (orientable or non-orientable according to the
context) Seifert fibered space whose orbit space is the surface
$S$ and having $n$ exceptional fibers, with non-normalized
parameters $(\alpha_i, \beta_i),\ i = 1, \ldots, n$;
\item[-] for each matrix $A\in GL (2;
\mathbb Z)$,\ $TB(A)= (T \times I)/A$ \ is the torus bundle over
$\mathbb S^1$ with monodromy induced by $A$;
\item[-] $ H_1 \bigcup_A H_2$ is the graph manifold
obtained by gluing a Seifert manifold $H_1$, with $\partial
H_1\cong T$, and a Seifert manifold $H_2$, with $\partial H_2\cong
T$, along their boundary tori by means of the attaching map
associated to matrix $A$;
\item[-] following \cite{M3}, $Q_i(p,q)$
denotes the closed manifold obtained as Dehn fil\-ling with
parameters $(p,q)$ of the compact manifold $Q_i$, whose interior
is one of the 11 hyperbolic manifolds of finite volume with a
single cusp and complexity at most three (see \cite{CHW} and
\cite{M2}).
\end{itemize}

In Table 2 we wrote in italics the 6-tuples representing manifolds
which don't appear in former catalogues of crystallizations
(\cite{BCrG},\cite{BGR},\cite{CC}).

\bigskip

\centerline{\begin{tabular}{|c|c|c|c|}
  \hline\ & \ & \ & \\
  \hfill  \textbf{tetrahedra} &
\textbf{6-tuple} & \textbf{$3$-manifold} & \textbf{position}\\ \ &
\ & \  &
\textbf{in \cite{KMN}}\\
 \hline\ & \ & \ & \\
18 & $(3,3,3,2,2,2)$ & $\mathbb S^3/Q_8=(\mathbb S^2,(2,1),(2,1),(2,-1))$ & \textbf{P.6}\\
   \hline\ & \ & \ & \\
  22 & $(3,3,5,2,2,4)$ & $\mathbb S^3/Q_{12}=(\mathbb S^2,(2,1),(2,1),(3,-2))$ & \textbf{P.14}\\
   \hline\ & \ & \ &\\
  24 & $(4,4,4,1,1,1)$ & $\mathbb S^3/Q_{8}\times Z_3=(\mathbb S^2,(2,1),(2,1),(2,1))$ & \textbf{P.25}\\
     & $(4,4,4,1,1,5)$ & $\mathbb S^3/D_{24}=(\mathbb S^2,(2,1),(2,1),(3,-1))$ & \textbf{P.29}\\
     & $(4,4,4,3,3,3)$ & $\mathbb S^3/P_{24}=(\mathbb S^2,(2,1),(3,1),(3,-2))$ & \textbf{P.11}\\
     \hline\ & \ & \ & \\
  26 & $(3,3,7,2,2,6)$ & $\mathbb S^3/Q_{16}=(\mathbb S^2,(2,1),(2,1),(4,-3))$ & \textbf{P.7}\\
 \hline\ & \ & \ & \\
  28 & $(4,4,6,1,1,1)$ & $\mathbb S^3/D_{48}=(\mathbb S^2,(2,1),(2,1),(3,1))$ & \textbf{P.50}\\
     & $(4,4,6,1,1,7)$ & $\mathbb S^3/P^\prime_{72}=(\mathbb S^2,(2,1),(3,1),(3,-1))$ & \textbf{P.34}\\
     & $(4,4,6,1,5,1)$ & $\mathbb S^3/Q_{16}\times Z_3=(\mathbb S^2,(2,1),(2,1),(4,-1))$ & \textbf{P.26}\\
\hline
\end{tabular}}

\vskip 20pt

\rightline{\it (Table 2 continues...)}


\centerline{\begin{tabular}{|c|c|c|c|}
  \hline\ & \ & \ & \ \\
  \hfill  \textbf{tetrahedra} &
\textbf{6-tuple} & \textbf{$3$-manifold} & \textbf{position}\\ \ &
\  & \  & \textbf{in \cite{KMN}}\\
\hline\ & \ & \ & \\
  28 & $(4,4,6,3,3,5)$ & $\mathbb S^3/P_{48}=(\mathbb S^2,(2,1),(3,1),(4,-3))$ & \textbf{P.3}\\
\hline\ & \ & \ &\\
  30 & $(3,3,9,2,2,8)$ & $\mathbb S^3/Q_{20}=(\mathbb S^2,(2,1),(2,1),(5,-4))$ & \textbf{P.15}\\
     & $(5,5,5,2,2,2)$ & $\mathbb E^3/P2_12_12_1=(\mathbb {RP}^2,(2,1),(2,-1))$ & \textbf{P.18}\\
     & $(5,5,5,4,4,4)$ & $\mathbb S^3/P_{120}=(\mathbb S^2,(2,1),(3,1),(5,-4))$ & \textbf{P.1}\\
\hline\ & \ & \ & \\
  32 & $(4,4,8,1,1,1)$ & $\mathbb S^3/Q_{16}\times Z_5=(\mathbb S^2,(2,1),(2,1),(4,1))$ & \textbf{P.40}\\
     & $(4,4,8,1,1,9)$ & $\mathbb S^3/P_{48}\times Z_5=(\mathbb S^2,(2,1),(3,2),(4,-3))$ & \textbf{P.38}\\
     & $(4,4,8,1,5,1)$ & $\mathbb S^3/D_{80}=(\mathbb S^2,(2,1),(2,1),(5,-1))$ & \textbf{P.51}\\
     & $(4,6,6,1,1,1)$ & $\mathbb S^3/P_{24}\times Z_7=(\mathbb S^2,(2,1),(3,1),(3,1))$ & \textbf{P.59}\\
     & $(4,6,6,1,1,9)$ & $TB\begin{pmatrix} -1 & 1\\ -1 & 0\end{pmatrix}=(\mathbb S^2,(3,1),(3,1),(3,-1))$ & \textbf{P.13}\\
     & $(4,6,6,1,7,1)$ & $\mathbb S^3/P_{48}\times Z_7=(\mathbb S^2,(2,1),(3,1),(4,-1))$ & \textbf{P.46}\\
     & $(4,6,6,5,5,3)$ & $\mathbb E^3/P4_1=TB\begin{pmatrix} 0 & 1\\ -1 & 0\end{pmatrix}=(\mathbb {S}^2,(2,1),(4,1),(4,-3))$ &
     \textbf{P.74}\\
\hline\ & \ & \ & \\
   34 & $(3,3,11,2,2,4)$ & $\mathbb S^3/D_{40}=(\mathbb S^2,(2,1),(2,1),(5,-3))$ & \textbf{P.30}\\
      & $(3,3,11,2,2,10)$ & $\mathbb S^3/Q_{24}=(\mathbb S^2,(2,1),(2,1),(6,-5))$ & \textbf{P.8}\\
      & $(3,7,7,2,2,2)$ & $\mathbb S^3/P_{24}\times Z_5=(\mathbb S^2,(2,1),(3,2),(3,-1))$ & \textbf{P.48}\\
      & $(5,5,7,2,4,2)$ & $(\mathbb {RP}^2,(2,1),(2,1))$ & \textbf{P.19}\\
      & $(5,5,7,2,6,6)$ & $\mathbb E^3/P3_1=TB\begin{pmatrix} 0 & 1\\ -1 & -1\end{pmatrix}=(\mathbb S^2,(3,1),(3,1),(3,-2))$ & \textbf{P.75}\\
      & $(5,5,7,4,4,6)$ & $\mathbb E^3/P6_1=TB\begin{pmatrix} 1 & -1\\ 1 & 0\end{pmatrix}=(\mathbb S^2,(2,1),(3,1),(6,-5))$ & \textbf{P.72}\\
    \hline\end{tabular}}

\vskip 20pt

\rightline{\it (Table 2 continues...)}


\centerline{\begin{tabular}{|c|c|c|c|}
  \hline\ & \ & \ & \\
  \hfill  \textbf{tetrahedra} &
\textbf{6-tuple} & \textbf{$3$-manifold} & \textbf{position}\\ &
 &  &
\textbf{in \cite{KMN}}\\
\hline\ & \ & \ & \\
  36 & ${\it (4,4,10,1,1,1)}$ & $\mathbb S^3/D_{40}\times Z_3=(\mathbb S^2,(2,1),(2,1),(5,1))$ & \textbf{P.63}\\
     & $(4,4,10,1,1,7)$ & $\mathbb S^3/Q_{12}\times Z_5=(\mathbb S^2,(2,1),(2,1),(3,2))$ & \textbf{P.57}\\
     & ${\it (4,4,10,1,1,11)}$ & $\mathbb S^3/P_{120}\times Z_{11}=(\mathbb S^2,(2,1),(3,2),(5,-4))$ & \textbf{P.42}\\
     & ${\it (4,4,10,1,5,1)}$ & $\mathbb S^3/Q_{24}\times Z_5=(\mathbb S^2,(2,1),(2,1),(6,-1))$ & \textbf{P.41}\\
     & $(4,4,10,1,5,7)$ & $\mathbb S^3/Q_{20}\times Z_3=(\mathbb S^2,(2,1),(2,1),(5,-2))$ & \textbf{P.43}\\
     & $(4,4,10,3,3,3)$ & $\mathbb S^3/P_{120}\times Z_7=(\mathbb S^2,(2,1),(3,1),(5,-3))$ & \textbf{P.28}\\
     & ${\it(4,6,8,1,1,1)}$ & $\mathbb S^3/P_{48}\times Z_{13}=(\mathbb S^2,(2,1),(3,1),(4,1))$ & \textbf{P.68}\\
     & ${\it(4,6,8,1,1,11)}$ & $(\mathbb S^2,(3,1),(3,2),(4,-3))$ & \textbf{P.35}\\
     & ${\it(4,6,8,1,7,1)}$ & $\mathbb S^3/P_{120}\times Z_{19}=(\mathbb S^2,(2,1),(3,1),(5,-1))$ & \textbf{P.56}\\
     & ${\it(4,6,8,3,9,13)}$ & $(\mathbb S^2,(2,1),(4,1),(4,-1))$ & \textbf{P.33}\\
     & $(4,6,8,5,5,11)$ & $(\mathbb S^2,(2,1),(4,1),(5,-4))$ & \textbf{P.4}\\
     & ${\it(6,6,6,1,1,1)}$ & $(\mathbb S^2,(3,1),(3,1),(3,1))$ & \textbf{P.37}\\
     & ${\it(6,6,6,1,1,9)}$ & $(\mathbb S^2,(3,1),(3,1),(4,-1))$ & \textbf{P.49}\\
     & $(6,6,6,1,7,7)$ & $TB\begin{pmatrix} -1 & 0\\ -1 & -1\end{pmatrix}=(K,(1,1))$ & \textbf{P.76}\\
     & $(6,6,6,5,5,5)$ & $TB\begin{pmatrix} 1 & 0\\ 1 & 1\end{pmatrix}=(T,(1,1))$ & \textbf{P.78}\\
\hline\ & \ & \ & \\
  38 & ${\it(3,3,13,2,2,12)}$ & $\mathbb S^3/Q_{28}=(\mathbb S^2,(2,1),(2,1),(7,-6))$ & \textbf{P.16}\\& $(5,5,9,2,2,2)$ & $(\mathbb {RP}^2,(2,1),(3,-1))$ & \textbf{P.66}\\
     & $(5,5,9,4,4,8)$ & $(\mathbb S^2,(2,1),(3,1),(7,-6))$ & \textbf{P.2}\\
     & $(5,7,7,4,6,12)$ & $(\mathbb S^2,(3,1),(3,1),(4,-3))$ & \textbf{P.12}\\
     \hline\end{tabular}}

\vskip 30pt

\rightline{\it (Table 2 continues...)}


\centerline{\begin{tabular}{|c|c|c|c|}
  \hline\ & \ & \ & \\
  \hfill  \textbf{tetrahedra} &
\textbf{6-tuple} & \textbf{$3$-manifold} & \textbf{position}\\ &
 &  &
\textbf{in \cite{KMN}}\\
\hline\ & \ & \ & \\
  40 & ${\it(4,4,12,1,1,1)}$ & $\mathbb S^3/Q_{24}\times Z_7=(\mathbb S^2,(2,1),(2,1),(6,1))$ & \textbf{P.47}\\
     & $(4,4,12,1,1,5)$ & $\mathbb S^3/Q_{8}\times Z_5=(\mathbb S^2,(2,1),(2,1),(2,3))$ & \textbf{P.39}\\& ${\it(4,4,12,1,1,13)}$ & $(\mathbb S^2,(2,1),(3,2),(6,-5))$ & \textbf{P.44}\\
     & ${\it(4,4,12,1,5,1)}$ & $\mathbb S^3/D_{56}\times Z_3=(\mathbb S^2,(2,1),(2,1),(7,-1))$ & \textbf{P.64}\\
     & ${\it(4,6,10,1,1,1)}$ & $\mathbb S^3/P_{120}\times Z_{31}=(\mathbb S^2,(2,1),(3,1),(5,1))$ & \textbf{P.70}\\
     & ${\it(4,6,10,1,1,13)}$ & $(\mathbb S^2,(3,1),(3,2),(5,-4))$ & \textbf{P.36}\\
     & ${\it(4,6,10,1,7,1)}$ & $(\mathbb S^2,(2,1),(3,1),(6,-1))$ & \textbf{P.65}\\
     & $(4,6,10,3,5,3)$ & $(\mathbb S^2,(2,1),(4,1),(5,-3))$ & \textbf{P.23}\\
     & ${\it(4,6,10,3,9,15)}$ & $(\mathbb S^2,(2,1),(4,3),(5,-4))$ & \textbf{P.55}\\
     & $(4,6,10,5,1,1)$ & $\mathbb S^3/P_{48}\times Z_{11}=(\mathbb S^2,(2,1),(3,2),(4,-1))$ & \textbf{P.61}\\
     & ${\it(4,6,10,5,5,13)}$ & $(\mathbb S^2,(2,1),(4,1),(6,-5))$ & \textbf{P.10}\\
     & $(4,6,10,5,9,3)$ & $(\mathbb S^2,(3,1),(3,2),(3,-1))$ & \textbf{P.27}\\
     & ${\it(4,6,10,7,1,1)}$ & $\mathbb S^3/P^\prime_{216}=(\mathbb S^2,(2,1),(3,1),(3,2))$ & \textbf{P.69}\\
     & $(4,6,10,7,3,15)$ & $\mathbb S^3/P_{120}\times Z_{13}=(\mathbb S^2,(2,1),(3,1),(5,-2))$ & \textbf{P.45}\\
     & ${\it(4,8,8,1,1,1)}$ & $(\mathbb S^2,(2,1),(4,1),(4,1))$ & \textbf{P.53}\\
     & ${\it(4,8,8,1,1,13)}$ & $(\mathbb S^2,(3,2),(4,1),(4,-3))$ & \textbf{P.32}\\
     & ${\it(4,8,8,1,9,1)}$ & $(\mathbb S^2,(2,1),(4,1),(5,-1))$ & \textbf{P.62}\\
     & $(4,8,8,5,5,13)$ & $(\mathbb S^2,(2,1),(5,1),(5,-4))$ & \textbf{P.20}\\
     & ${\it(6,6,8,1,1,1)}$ & $(\mathbb S^2,(3,1),(3,1),(4,1))$ & \textbf{P.71}\\
     & ${\it(6,6,8,1,1,11)}$ & $(\mathbb S^2,(3,1),(4,1),(4,-1))$ & \textbf{P.52}\\
     & ${\it(6,6,8,1,9,1)}$ & $(\mathbb S^2,(3,1),(3,1),(5,-1))$ & \textbf{P.60}\\
     \hline\end{tabular}}

\vskip 20pt

\rightline{\it (Table 2 continues...)}


\centerline{\begin{tabular}{|c|c|c|c|}
  \hline\ & \ & \ & \\
  \hfill  \textbf{tetrahedra} &
\textbf{6-tuple} & \textbf{$3$-manifold} & \textbf{position}\\ &
 &  &
\textbf{in \cite{KMN}}\\
 \hline\ & \ & \ & \\
  40 & $(6,6,8,5,5,7)$ & $TB\begin{pmatrix} 0 & 1\\ -1 & 3\end{pmatrix}=Q_2(0,1)$ & \textbf{P.73}\\
     & $(6,6,8,5,11,7)$ & $TB\begin{pmatrix} 0 & 1\\ -1 & -3\end{pmatrix}=Q_1(1,1)$ & \textbf{P.77}\\
\hline\ & \ & \ & \\
  42 & $(3,3,15,2,2,6)$ & $\mathbb S^3/D_{56}=(\mathbb S^2,(2,1),(2,1),(7,-5))$ & \textbf{P.31}\\
     & ${\it(3,3,15,2,2,14)}$ & $\mathbb S^3/Q_{32}=(\mathbb S^2,(2,1),(2,1),(8,-7))$ & \textbf{P.9}\\
     & $(3,7,11,4,2,2)$ & $\mathbb S^3/P_{120}\times Z_{17}=(\mathbb S^2,(2,1),(3,2),(5,-3))$ & \textbf{P.54}\\
     & $(5,5,11,2,4,2)$ & $(\mathbb {RP}^2,(2,1),(3,1))$ & \textbf{P.67}\\
     & $(5,5,11,4,4,10)$ & $(\mathbb S^2,(2,1),(3,1),(8,-7))$ & \textbf{P.5}\\
     & $(5,5,11,4,8,4)$ & $(\mathbb S^2,(2,1),(3,1),(7,-5))$ & \textbf{P.21}\\
     & $(5,7,9,2,4,4)$ & $(\mathbb D,(2,1),(2,-3))\bigcup_{\begin{pmatrix}0 & 1\\1 & 0\end{pmatrix}}(\mathbb D,(2,1),(3,-2))$ & \textbf{P.58}\\
     & ${\it(5,7,9,4,6,14)}$ & $(\mathbb S^2,(3,1),(3,1),(5,-4))$ & \textbf{P.24}\\
     & $(7,7,7,2,2,2)$ & $Q_1(2,-3)$ & \textbf{P.22}\\
     & $(7,7,7,2,6,10)$ & $(\mathbb D,(2,1),(2,1))\bigcup_{\begin{pmatrix}0 & 1\\1 & 0\end{pmatrix}}(\mathbb D,(2,1),(3,1))$ & \textbf{P.17}\\
\hline \end{tabular}}

\vskip 30pt

\centerline{\bf TABLE 2:  Prime genus two 3-manifolds represented
by crystallizations of $\textbf C_2^{(42)}$}

\bigskip

 \centerline{ \begin{tabular}{|c|c|}
  \hline
  \hfill  \textbf{tetrahedra} \hfill &
\textbf{$3$-manifold }  \\
 \hline \ & \ \\
  16 & $\mathbb {RP}^2\times\mathbb S^1$ \\
  \hline \ & \ \\
  32 & $\mathbb E^3 /Bb=TB\begin{pmatrix} 0 & 1\\ 1 & 0\end{pmatrix}$
 \\ \ & \ \\
 & $\mathbb E^3 /Pna2_1=(\mathbb R \mathbb P^2;
 (2,1),(2,1))$  \\ \ & \
 \\\hline \ & \ \\
 34 &  $TB\begin{pmatrix} 0 & 1\\ 1 & -1\end{pmatrix}$ \\ \ & \ \\
 \hline \ & \ \\ 36 & $TB\begin{pmatrix} 2 & 1\\ 1 &
0\end{pmatrix}$
 \\ \ & \  \\
 \hline \ & \ \\ 40  & $TB\begin{pmatrix} 3 & 1\\ 1 & 0\end{pmatrix}$\\ \ & \ \\
 & ($\mathbb R \mathbb P^2;
 (2,1),(3,1)$)
\\ \ & \ \\
\hline \end{tabular}}

\bigskip
\centerline{\bf TABLE 3:  Prime genus two 3-manifolds represented
by crystallizations of $\tilde{\textbf C_2}^{(42)}$}

\end{document}